\documentclass[12pt]{article}
\usepackage{amsmath,amssymb,latexsym,amsthm}

\usepackage[english]{babel}
\usepackage{url}
\usepackage{ulem}
\usepackage[pdftex,breaklinks,colorlinks,linkcolor=blue,
anchorcolor=blue]{hyperref}

\usepackage{color}

\newtheorem{defn}{Definition}[section]
\newtheorem{lemma}[defn]{Lemma}
\newtheorem{ex}[defn]{Example}
\newtheorem{thm}[defn]{Theorem}
\newtheorem{prop}[defn]{Proposition}
\newtheorem{cor}[defn]{Corollary}
\newtheorem{conj}[defn]{Conjecture}
\newtheorem*{conj1}{Conjecture (M1)}
\newtheorem*{conj2}{Conjecture (M2)}

\theoremstyle{remark}
\newtheorem{rem}[defn]{Remark}

\newcommand{\h}{{\cal H}}

\newcommand{\mc}{\mathbb C}
\def\dom{{\cal D}}
\def\<{\langle}
\def\>{\rangle}

\newcommand{\ip}[2]{
	\< #1,#2 \>}

\definecolor{darkviolet}{rgb}{0.58,0,0.83} 

\definecolor{cel}{RGB}{0,140,20}

\usepackage{cite}

\title{Weighted frames, weighted lower semi frames and unconditionally convergent multipliers}
\author{P. Balazs$^{a,1}$, R. Corso$^{b,2}$,  D. Stoeva$^{c,3}$}
\date{%
	{\small 
	$^a$Acoustics Research Institute, Austrian Academy of Sciences, Wohllebengasse 12-14, 1040 Vienna, Austria\\%
	$^b$Dipartimento di Matematica e Informatica, Universit\`{a} degli Studi di Palermo, Via Archirafi 34, \\90123 Palermo, Italy\\
	$^c$Faculty of Mathematics, University of Vienna, Oskar-Morgenstern-Platz 1, 1090 Vienna, Austria
	$^1$peter.balazs@oeaw.ac.at, $^2$rosario.corso02@unipa.it, $^3$diana.stoeva@univie.ac.at \\
	}
	\vspace*{0.3cm }
	October 21, 2023
	}

\begin{document}
\maketitle \pagestyle{myheadings}

\begin{abstract} 
	In this paper we ask when it is possible to transform a given sequence into a frame or a lower semi frame by multiplying the elements by numbers. In other words, we ask when a given sequence is a weighted frame or a weighted lower semi frame and for each case we formulate a conjecture. We determine several conditions under which these conjectures are true. Finally, we prove an equivalence between two older conjectures, the first one being that any unconditionally convergent multiplier can be written as a multiplier of Bessel sequences by shifting of weights,  and the second one that every unconditionally convergent multiplier which is invertible can be written as a  multiplier of frames by shifting of  weights. We also show that these conjectures are also related to one of the newly posed conjectures. 
	
\end{abstract}

\noindent {\bf Keywords:} weighted frames, weighted lower semi frames, multipliers, unconditionally convergent multipliers.

\section{Introduction}

Let $\h$ be an infinite dimensional separable Hilbert space with inner product $\ip{\cdot}{\cdot}$ and norm $\|\cdot\|$. Throughout the paper we  use $\mathbb{N}$ as an index set, also implicitly. 
Let $(f_n)$ be a sequence with elements of $\h$. Recall that $(f_n)$ is called a {\it frame for $\h$} if there exist positive constants $A$ and $B$ such that
\begin{equation}\label{framedef}
A\|f\|^2 \leq \sum_{n} |\ip{f}{f_n}|^2\leq B\|f\|^2, \qquad \forall f\in \h.  
\end{equation}
Frames have been introduced by Duffin and  Schaffer \cite{DS} and they provide reconstruction formulas on the entire Hilbert space. More precisely, if $(f_n)$ is a frame for $\h$, then there exists a frame $(g_n)$ for $\h$ such that
$$
f=\sum_n \ip{f}{f_n}g_n=\sum_n \ip{f}{g_n}f_n, \qquad \forall f\in \h, 
$$
see \cite{DS}. 
Separate consideration of the inequalities in (\ref{framedef}) has also been of interest. 
The sequence $(f_n)$ is a {\it Bessel sequence in $\h$} with  Bessel bound $B\in(0,\infty)$ if the upper inequality in (\ref{framedef}) holds for every $f\in\h$. 
For more on frames and Bessel sequences, we refer e.g. to the books \cite{ole1n,heil}. 
The sequence $(f_n)$ is called a {\it lower semi frame for $\h$} with bound $A>0$ if the lower inequality in (\ref{framedef}) holds for every $f\in\h$. 
A lower semi frame $\mathcal{F}=(f_n)$
for $\h$ provides a reconstruction formula 
of the type  
\begin{equation}\label{reprlf}
f=\sum_n \ip{f}{f_n}g_n, \qquad \forall f\in \dom(C_\mathcal{F}), 
\end{equation}
where $(g_n)$ is some Bessel sequence in $\h$ and $\dom(C_\mathcal{F})$ is the domain of the analysis operator\footnote{See Section \ref{sec2} for the definition of the analysis operator.} of $(f_n)$. This result dates back to \cite{CCLL02}. 
Contrary to the frame case, where reconstruction of all the elements of the Hilbert space is guaranteed, for certain lower semi frames 
a reconstruction formula may hold only on a strict subset of $\h$, 
see \cite{CCLL02,S05}. For more on lower frame sequences, we refer to \cite{AB,AB2,CCLL02,Corso_seq,Corso_seq2,S05,Sd}. 
For the purpose of a notion symmetrical to the lower semi frame one, an upper semi-frame for $\h$ has been introduced \cite{AB} to be a Bessel sequence in $\h$ that is furthermore complete in $\h$.

\vspace{.05in}
 This paper is centered on the following notions. 

\begin{defn}
	\label{def_wf}
	A sequence $(f_n)$ in $\h$ is called a weighted frame\footnote{Weighted frames were previously introduced and studied in \cite{Jacques,BVAJM,xxljpa1}. The notion is also connected to  scalable frames \cite{kopt}.} (resp. weighted lower semi frame) for $\h$ if there exists $(\omega_n)\subset (0,+\infty)$ such that $(\omega_n f_n)$ is a frame (resp. lower semi frame) for $\h$. 
\end{defn}

Specifically, in this paper we study the following problems

\begin{enumerate}
	\item[P1)] Under which conditions is a given sequence a weighted lower semi frame for $\h$?
	\item[P2)] Under which conditions is a given sequence a weighted frame for $\h$?
\end{enumerate}

Further introductory definitions and results are given in Sect. \ref{sec2}. Weighted lower semi frames and Problem P1 are studied in Sect. \ref{sec3}. Similarly, we dedicate Sect. \ref{sec4} to weighted frames and to Problem P2. 
At the end of Sect. \ref{sec4} we relate the topic to unconditionally convergent multipliers. In particular, we establish the equivalence between two conjectures put forward in \cite{uncconv2011}. 
	 We show the surprising fact that if it is true that any unconditionally convergent  invertible multiplier is a multiplier of frames by shifts of weights, then it is also true that any unconditionally convergent multiplier is a multiplier of Bessel sequences by shifts of weights.

\section{Preliminaries}
\label{sec2}

Given an operator $T$ on a vector space we write $\mathcal{D}(T)$ for its domain  and $R(T)$ for its range. As usual, the symbol $\mathcal{B}(\h)$ stands for the set of all bounded operators with domain $\h$ and $GL(\h)$ denotes the set of all  bounded bijective operators 
from $\h$ onto $\h$. 
A sequence $(f_n)$ of elements of $\h$ is {\it complete in $\h$} if its linear span is dense in $\h$. A sequence $(f_n)$ in $\h$ is called {\it minimal} if $f_j \notin$ $\overline{\text{span}}(f_n)_{n \neq j}$ for all $j$. As it is well known, $(f_n)$ is minimal if and only if it is {\it biorthogonal} to some sequence $(g_n)$, i.e. $\ip{f_n}{g_m}=\delta_{n,m}$; a minimal sequence that is complete in $\h$ has a unique biorthogonal in $\h$.
Throughout the paper, we will make use of two operators associated to a sequence $\mathcal{F}:=(f_n)$. The {\it analysis operator} $C_\mathcal{F} :\dom(C_\mathcal{F} )\subseteq \h\to \ell^2$ is given by 
$\dom(C_\mathcal{F})=\{f\in \h: \sum_{n} |\ip{f}{f_n}|^2<\infty \}$ and 
$C_\mathcal{F} f=(\ip{f}{f_n})$, for 
$f\in\dom(C_\mathcal{F})$. The {\it synthesis operator} $D_\mathcal{F}:\dom(D_\mathcal{F})\subseteq \ell^2 \to \h$ is given by 
$\dom(D_\mathcal{F})=\left \{ (c_n)\in \ell^2:\sum_n c_n f_n \text{ converges in } \h\right \}$  
and $D_\mathcal{F}(c_n)=\sum_n c_n f_n$, for $ (c_n)\in \dom(D_\mathcal{F})$. 
The operator {$C_\mathcal{F}$} is always closed \cite{CCLL02} and $C_\mathcal{F}=D_\mathcal{F}^*$ \cite{ABS}.
Detailed properties of the analysis and synthesis operators (alongside other operators associated to sequences) can be found in \cite{ABS}.

Let us recall some basic  
results from frame theory (that can be found e.g. in \cite{ole1n}). 
For a Bessel sequence $(f_n)$ and $(c_n)\in \ell^2$ the series $\sum_{n=1}^\infty c_n f_n$ is always convergent. 
A {\it Riesz basis} $(f_n)$ for $\h$ is a complete sequence in $\h$ satisfying 
$$
A \sum_{n} |c_n|^2 \leq \left \| \sum_{n} c_n f_n \right \|^2 \leq B\sum_{n} |c_n|^2, \qquad\forall (c_n)\in \ell^2,
$$
for some positive constants $A$ and $B$. A sequence is a Riesz basis for $\h$ if it is a frame for $\h$ that has a biorthogonal sequence. 
As done for frames and lower semi frames, we say that a sequence $(f_n)$ is a {\it weighted Riesz basis} for $\h$ if there  exists $(\omega_n)\subset (0,+\infty)$ such that $(\omega_n f_n)$ is a Riesz basis for $\h$.
Note that Problem P2 with respect to weighted Riesz bases is related to the following well-known characterization of Riesz bases. 

\begin{thm} \label{the:rieszclass} \cite[Lemma 3.6.9]{ole1n} A sequence $( f_n )$ is a Riesz basis for $\h$ if and only if it is an unconditional basis for $\h$ and it is norm-semi-normalized (i.e. $0 < \inf \| f_n \| \le \sup \| f_n \| < \infty$). In other words, a sequence $( f_n )$ is a weighted Riesz basis for $\h$ if and only if it is an unconditional basis for $\h$. 
\end{thm}

In \cite{spexxl14,speckB17} reproducing pairs have been introduced and studied as a way to go beyond the frame reconstruction formulas.
A pair of sequences $((f_n),(g_n))$ is called a {\it reproducing pair for $\h$} if the operator $T$ defined weakly by $\<Tf,g\>=\sum_{n} \ip{f}{f_n}\ip{g_n}{g}$ belongs to $GL(\h)$. Reproducing pairs are related to the more general notion of sesquilinear forms. For study on sesquilinear forms, we refer to \cite{Kato,Corso_seq,Corso_seq2}.

\section{Weighted lower semi frames}
\label{sec3}

Since a weighted frame is, in particular, a weighted lower semi frame, we start by analyzing sequences that can be transformed into lower semi frames. 
First, observe that if a sequence is a weighted lower semi frame for $\h$, then it is complete in $\h$.
We conjecture that completeness is furthermore a sufficient condition for the lower semi frame property:

\begin{conj}
	\label{conjA}
	Let $(f_n)$ be a complete sequence in $\h$ and $A\in(0,\infty)$. There exists $(\omega_n)\subset (0,\infty)$ such that $(\omega_n f_n)$ is a lower semi frame for $\h$ with a bound $A$.
\end{conj}

In this section we determine some classes  of sequences for which Conjecture \ref{conjA} holds.  
First of all, we prove that the corresponding problem for Bessel sequences can be easily solved and it is helpful for dealing with special cases of the above conjecture.

\begin{lemma}
	\label{lem_Bessel}
	Let $(f_n)$ be a sequence in $\h$. For any $B>0$ there exists a sequence $(\lambda_n)\subset(0,\infty)$ such that $(\lambda_n f_n)$ is a Bessel sequence with bound $B$.
\end{lemma}
\begin{proof}
	Let $B>0$ and let $(\tau_n)\subset(0,\infty)$ be a sequence such that $\sum_n \tau_n^2=1$.  Define  $\lambda_n=\tau_n\sqrt{B}\|f_n\|^{-1}$ if $f_n\neq0$ and  $\lambda_n =1$ (or any other positive number) if $f_n=0$. 
	Then, by the Cauchy-Schwarz inequality, we have for all $f\in \h$ that
	\[
	\sum_n |\<f,\lambda_n f_n\>|^2\leq \sum_n \tau_n^2B\|f\|^2=B\|f\|^2. \qedhere
	\]
\end{proof}

With this lemma we can then prove our first result. 

\begin{prop} 
	\label{pro_dualities}
Let  $((f_n), (g_n))$    be a reproducing pair for $\h$. 
Then there are weights $(\lambda_n)\subset (0,\infty)$ and $(\beta_n)\subset (0,\infty)$ 
such that $(\frac{1}{\lambda_n} f_n)$ and $(\frac{1}{\beta_n} g_n)$ are lower semi-frames for $\h$,  and $(\lambda_n g_n)$ and $(\beta_n f_n)$ are 
upper semi-frames for $\h$. 
In particular, Conjecture \ref{conjA} is true for $(f_n)$ and $(g_n)$. 
\end{prop}
\begin{proof}
	Let $T\in GL(\h)$ be the operator such that $\<Tf,g\>=\sum_n \<f,f_n\>\<g_n,g\>$ for all $f,g\in \h$. Then
	\begin{equation}\label{reprpair2}
	\< f,g\>=\sum_n \<f,f_n\>\< h_n,g\>, \qquad \forall f,g\in \h,
	\end{equation}
	where $(h_n)=(T^{-1} g_n)$.
	
	Let $A>0$. By Lemma \ref{lem_Bessel}, there exists a sequence $(\lambda_n)\subset(0,+\infty)$ such that $(\lambda_n h_n)$ is Bessel with bound $A^{-1}$. 
	Since
	$$
		\< f,g\>=\sum_n \<f,\lambda_n^{-1} f_n\>\< \lambda_n h_n,g\>, \qquad \forall f,g\in \h,
		$$
using the Cauchy-Schwarz inequality, similar to \cite[Lemma 6.3.2]{ole1n} we get that for every $f\in\h$, 
		$$\|f\|^4\leq (\sum_{n} |\<f,\lambda_n^{-1} f_{n}\>|^2)(\sum_{n} |\<f,\lambda_n h_n\>|^2)\leq A^{-1}\|f\|^2\sum_{n} |\<f,\lambda_n^{-1} f_{n}\>|^2. $$ 	
	Thus, $(\lambda_n^{-1} f_{n})$ is a lower semi frame for $\h$ with bound $A$. Since 
	$(\lambda_n h_n)$ is a Bessel sequence, clearly $(\lambda_n g_n)$ is also Bessel. 
	
	Further, by Lemma \ref{lem_Bessel}, take a sequence $(\beta_n)\subset(0,+\infty)$ such that $(\beta_n f_n)$ is Bessel with bound $B=A^{-1}\|T^{-1}\|^{-2}$.
	Similar to the above paragraph, using the Cauchy-Schwarz inequality, we get that the sequence 
	$(\beta_n^{-1} h_{n})$ is a lower semi frame for $\h$ with bound $B^{-1}$. 
	 Then $(\beta_n^{-1} g_{n})$ is a lower semi frame for $\h$ with bound $B^{-1}\|T^{-1}\|^{-2}=A$. 
	
	Since $(\lambda_n^{-1} f_n)$ and $(\beta_n^{-1} g_n)$ are complete in $\h$, clearly $(\lambda_n g_n)$ and $(\beta_n f_n)$ are also complete in $\h$. 
\end{proof}

As particular case, Proposition \ref{pro_dualities} 
applies to sequences $(f_n)$ and $(g_n)$ which are {\it weakly dual} to each other in the sense that
	\begin{equation}
		\label{weaklydual}
	\< f,g\>=\sum_n \<f,f_n\>\<g_n,g\>, \qquad \forall f,g\in \h,
	\end{equation}
	or when the operator  $M$ given by 	
	\begin{equation} \label{multiplier1} M f=\sum_n \<f,g_n\> f_n 	\end{equation}
	is well defined on $\h$ and belongs to $GL(\h)$. Note that an operator $M$ determined by \eqref{multiplier1} is 
	a multiplier with a constant symbol (see Section \ref{sec_multipliers} for the definition of multipliers); see \cite{SB2012, SB2013,SB2020} 
	for results related to the question when a multiplier is in $GL(\h)$. 
For studies of representations of unbounded operators like  \eqref{multiplier1} we refer to \cite{BC}.

In particular, by the proofs of Proposition  \ref{pro_dualities} and	Lemma \ref{lem_Bessel}, the following can be stated for the important case of weakly dual sequences:
		
\begin{cor} \label{cor:reweightdual2} 
Let $((f_n), (g_n))$ be a weakly dual pair for $\h$, i.e., 
\eqref{weaklydual} holds. 
Then there are weights $(\lambda_n)\subset (0,\infty)$ and $(\beta_n)\subset (0,\infty)$ such that $(\frac{1}{\lambda_n} f_n, \lambda_n g_n)$ and $(\beta_n  f_n, \frac{1}{\beta_n} g_n)$ are weakly dual pairs for $\h$, 
 $(\frac{1}{\lambda_n} f_n)$ and $(\frac{1}{\beta_n} g_n)$ are lower semi-frames for $\h$,  and $(\lambda_n g_n)$ and $(\beta_n f_n)$ are 
 upper semi-frames for $\h$. 
 The weights can be chosen e.g. as  $\lambda_n=\tau_n\|g_n\|^{-1}$ if $g_n\neq 0$,  $\lambda_n=1$ if $g_n=0$, $\beta_n=\tau_n\|f_n\|^{-1}$ if $f_n\neq 0$, and $\beta_n=1$ if $f_n=0$, where $(\tau_n)\subset(0,\infty)$ is a sequence such that $\sum_n \tau_n^2=1$.
\end{cor}

As another direct application of Proposition \ref{pro_dualities} we have the following statement for Schauder bases\footnote{Recall that a sequence $(f_n)$ in $\h$ is a Schauder basis for $\h$ if there exists a unique sequence $(g_n)$ in $\h$ such that $f=\sum_n \ip{f}{g_n} f_n$ for every $f\in \h$.}.

\begin{cor}
	Let $(f_n)$ be a sequence in $\h$ containing a Schauder basis for $\h$ or a frame for $\h$. Then Conjecture \ref{conjA} is true for $(f_n)$. 
\end{cor}

Recall that when $(f_n)$ is a sequence in $\h$ whose synthesis operator is closed and surjective, then there is a Bessel sequence $(g_n)$ such that $f=\sum\<f, g_n\>f_n$ for every $f\in\h$ and $(f_n)$ satisfies the lower frame condition for $\h$, see \cite[Theorem 8.4.1]{ole1n}. Let us now consider the following (more general) case 
 using closed and surjective operator of the form \eqref{multiplier1}:

\begin{prop}\label{propclosedsurj}	Let $(f_n)$ and $(g_n)$ be sequences in $\h$. Define the operator $M$ by (\ref{multiplier1}) 
	with domain determined by those $f$ for which the series in \eqref{multiplier1} converges.
	If $M$ is closed and $R(M)=\h$, then Conjecture \ref{conjA} is true for $(f_n)$. 
\end{prop}
\begin{proof}
	By \cite[Lemma 1.1, Corollary 1.2]{BR}, the operator $M$  admits a pseudo-inverse $M^{+}$ which belongs to $\mathcal{B}(\h)$ since $R(M)=\h$. 
	This means, in particular, that 
	$M M^{+}h=h$ for every $h\in \h$. 
	 Then 
	$$
	h=M M^{+}h=\sum_n \ip{M^+h}{g_n}f_n =\sum_n \ip{h}{M^{+*}g_n}f_n, \qquad \forall h\in \h.
	$$
	By Corollary \ref{cor:reweightdual2}, Conjecture \ref{conjA} is true for $(f_n)$.
\end{proof}

We proceed with one more class of sequences for which Conjecture \ref{conjA} holds.

\begin{prop}
	\label{D(C)_finite}
	Any complete sequence in $\h$ whose analysis operator has finite dimensional domain is a lower semi frame for $\h$ (thus, Conjecture \ref{conjA} trivially holds for it). 
\end{prop}

\begin{proof} 
	Let $(f_n)$ be a complete sequence in $\h$ such that its analysis operator $C$ has finite dimensional (closed) domain $\dom(C)$. Then for $f\in \dom(C)$ we have
	\begin{equation}
		\label{eq1}
		\sum_{n=1}^\infty |\<f,f_n\>|^2=\sum_{n=1}^\infty |\<f,Pf_n\>|^2,
	\end{equation} where $P$ is the orthogonal projection onto $\dom(C)$. The sequence $(Pf_n)$ is complete in $\dom(C)$ and since $\dim \dom(C)=d<\infty$, there exist 
	$f_{n_1},f_{n_2}, \dots, f_{n_d}$ such that $Pf_{n_1},Pf_{n_2}, \dots, Pf_{n_d}$ are linearly independent and thus a Riesz basis for $\dom(C)$. This implies that $(Pf_n)$ is a lower semi frame for $\dom(C)$. 
	By \eqref{eq1}, $(f_n)$ satisfies the lower frame condition for $\dom(C)$ and hence $(f_n)$ is a lower semi frame for $\h$. 
\end{proof}

Note that if the analysis operator $C$ of $(f_n)$ has finite dimensional domain then also the orthogonal complement 
of $(f_n)$ is finite dimensional. This follows by the fact that if $h\perp f_n $ for every $n$, then $h\in \dom(C)$.

\begin{rem} 
As it can be expected, the above propositions do not cover all the possible classes of sequences for which Conjecture \ref{conjA} holds.  
	As an illustration, consider for example the sequence $(f_n)_{n=2}^\infty=(n(e_1+e_n))_{n=2}^\infty$ where $(e_n)_{n=1}^\infty$ denotes an orthonormal basis for $\h$. It is shown in \cite{S05} that $(f_n)_{n=2}^\infty$ satisfies the lower frame condition for $\h$,  
		but not all the elements of $\h$ can be represented in the form $\sum_n c_n f_n$ (hence,  Proposition \ref{propclosedsurj} does not apply) and the representation 
\eqref{reprlf} can not be extended to the entire space $\h$. In a similar way as in \cite{S05} it can be shown that $(f_n)$ 
	 does not generate a reproducing pair $(f_n, g_n)$  for $\h$ and thus Proposition \ref{pro_dualities} does not apply. Clearly, the domain of the analysis operator of $(f_n)$ is not finite dimensional, so Proposition \ref{D(C)_finite} does not apply either. 
\end{rem}

The above remark motivates interest to further consideration of Conjecture \ref{conjA}. 
Let us end this section showing that we can reduce the conjecture studying only sequences with densely defined analysis operators.

\begin{thm}
	The following statements are equivalent. 
	\begin{enumerate}
		\item Conjecture \ref{conjA} is true. 
		\item 
		 Conjecture \ref{conjA} is true for every complete sequence in $\h$ whose analysis operator is densely defined\,\footnote{i.e., whose synthesis operator is closable \cite{ABS}; in such case one also has $D \subseteq C^*$ \cite{Ole1995}.}. 
\end{enumerate}
\end{thm}
\begin{proof}
	$1.\implies 2.$ Clear.\\
	$2.\implies 1.$ Let $\mathcal{F}:=(f_n)$ be a complete sequence in $\h$ with analysis operator $C_\mathcal{F}$. 
We will prove that Conjecture \ref{conjA} holds for $F$  based on the assumption for validity of $2.$ 
	Let $W=\overline{\dom(C_\mathcal{F})}$ be the closure of $\dom(C_\mathcal{F})$ in $\h$ and let $P_W$ be the orthogonal projection onto $W$.  If $\dim W<\infty$, then the conjecture is true for $\mathcal{F}$ by Proposition \ref{D(C)_finite}. Assume that $W$ has infinite dimension. Recall that $\h$ is separable, therefore $W$ is separable and there exists a unitary operator $J:W\to \h$ (i.e., $J^*=J^{-1}$).  	Define $(h_n):=(P_W f_n)$ and $\mathcal{G}:=(g_n):=(Jh_n)$. The sequences $(h_n)$ and $(g_n)$ are complete in $W$ and $\h$, respectively. 
Let $C_{\mathcal{G}}:\dom(C_{\mathcal{G}})\subset \h\to \ell^2$ be the analysis operator of $\mathcal{G}$. Observe that for $f\in W$ the following relations hold
	$$
	\sum_n |\<Jf,g_n \>|^2=\sum_n |\<f,h_n \>|^2=\sum_n |\<P_W f,f_n \>|^2=\sum_n |\<f,f_n \>|^2,
	$$
	which implies that  $Jf\in \dom(C_{\mathcal{G}})$ if and only if $f\in \dom(C_{\mathcal{F}})$.	
	Therefore, $ J\dom(C_\mathcal{F})=\dom(C_\mathcal{G})$, leading to the conclusion that $\dom(C_{\mathcal{G}})$ is dense in $\h$.

	Fix an arbitrary $A>0$. Since $(g_n)$ is complete in $\h$ with $\dom(C_{\mathcal{G}})$ being dense in $\h$, by assumption there is a sequence $(\omega_n)\subset (0,\infty)$ 
	such that $(\omega_n g_n)$ is a lower semi frame for $\h$ with bound $A$. 
	By making the substitution $\omega_n \to \omega_n +1$, we can assume 
	 that $\omega_n>1$ for every $n$. 
	 It remains to prove that $(\omega_n f_n)$ is a lower semi frame for $\h$ with bound $A$. 

	Let $f\in \h$. If $ \sum_{n} |\<f,\omega_n f_n\>|^2=\infty$, then the lower frame inequality trivially holds. 
	Now consider the remaining case. In this case we have   
	$$
	\sum_{n} |\<f,f_n\>|^2\leq \sum_{n} |\<f,\omega_nf_n\>|^2<\infty,
	$$
	so $f$ is necessarily in $\dom(C_\mathcal{F})\subset W$. 
		Therefore,  
$$
	\sum_{n} |\<f,\omega_n f_n\>|^2=\sum_{n} |\<P_Wf,\omega_n f_n\>|^2
	=\sum_{n} |\<f,\omega_nh_n\>|^2
	=  \sum_{n} |\<Jf,\omega_n g_n\>|^2 \geq A \|Jf\|^2=A\|f\|^2. \qedhere
$$
\end{proof}

In \cite{bs} a similar approach looking at the closure of the domain of the analysis operator is used for results of lower frames sequences and Riesz-Fisher sequences.

\section{Weighted frames}
\label{sec4}

Note that in the literature \cite{Jacques,BVAJM,xxljpa1} Definition \ref{def_wf} was used to define the concept of $w$-frame while weighted frames were used either with the meaning of a $w$-frame or with the meaning of a sequence which can be transformed to a frame via a weight $ (w_n)\subset \mathbb{C}$. 
Trivially, Definition \ref{def_wf} is equivalent to a definition which requires $(\omega_n)\subset \mathbb{C}$ and $\omega_n\neq 0$ for every $n$ (taking $(|\omega_n|)$ in consideration). 
Notice that Lemma \ref{lem_Bessel} can bring conclusion for equivalence of Definition \ref{def_wf} to a definition that furthermore allows zero elements in $(w_n)$ (so, one can also use $[0,+\infty]$ or the entire set $\mathbb{C}$ for an equivalent definition).    
Indeed, assume that $(f_n)$ is a sequence such that $(\omega_n' f_n)$ is a frame for $\h$ for some $(\omega_n')\subset [0,\infty)$. Denote $I=\{n\in \mathbb{N}: \omega_n'\neq 0\}$ and split $(f_n)=(f_{m})_{m\in I}\cup (f_{k})_{k\notin I}$. By Lemma \ref{lem_Bessel}, there exists a sequence $(\tau_{k})_{k\notin I}\subset (0,\infty)$ such that $(\tau_{k} f_{k})_{k\notin I}$ is a Bessel sequence. Now, define $\omega_n=\omega_n'$ if $n\in I$ and $\omega_n=\tau_n$ if $n\notin I$. By construction, $(\omega_n)\subset (0,+\infty)$. Since $(\omega_m'f_m)_{m\in I}$ is a frame for  $\h$ and $(\tau_{k} f_{k})_{k\notin I}$ is a Bessel sequence, the sequence $(\omega_n f_n)$ is a frame for $\h$.

\subsection{Necessary conditions}

We begin with some necessary conditions that are easy consequences of known results about frames:

\begin{cor}
	\label{duality_weighted} 
	If $(f_n)$ is a weighted frame for $\h$, then
 $(f_n)$ is complete in $\h$ and there exists a sequence $(g_n)$ in $\h$ such that 
		$$
		f=\sum_n \<f,f_n\> g_n =\sum_n \<f,g_n\> f_n , \qquad \forall f\in \h,
		$$
		with unconditional convergence. 
\end{cor}

As in the case with frames, where completeness is necessary and not sufficient condition when the space is infinite dimensional, here completeness is also necessary and not sufficient for the weighted frame property. An illustration is given in the following example.

\begin{ex}
	\label{exm_e1+en}
	Let $(f_n)_{n=2}^\infty=(e_1+e_n)_{n=2}^\infty$, where $(e_n)_{n=1}^\infty$ is an orthonormal basis for $\h$. It is easy to see that  $(f_n)_{n=2}^\infty$ is complete in $\h$. However, $(f_n)_{n=2}^\infty$ is not a weighted frame for $\h$. Indeed, with an argument similar to the one used in \cite{S05}, there does not exist a sequence $(g_n)_{n=2}^\infty$ such that 
	$$
	f=\sum_n \<f,f_n\> g_n, \qquad \forall f\in \h.
	$$	
Then, by Proposition \ref{duality_weighted}, $(f_n)_{n=2}^\infty$ is not a weighted frame for $\h$.
\end{ex}

The next statement gives a class of complete sequences that are not weighted frames. For recent results on completeness in relation to the Bessel and Riesz basis properties we reefer to \cite{S2020}. 

\begin{prop}
	\label{weighted_biortho}
	Let $(f_n)$ be a complete minimal sequence in $\h$ whose biorthogonal sequence is not complete in $\h$. Then $(f_n)$ is not a weighted frame for $\h$. 
\end{prop}
\begin{proof}
	Suppose that $(f_n)$ is a weighted frame for $\h$ and a minimal sequence. In particular, this means that there exists $(\omega_n)\subset (0,+\infty)$ such that $(\omega_n f_n)$ is a Riesz basis for $\h$. Then the biorthogonal sequence to $(\omega_n f_n)$ is a Riesz basis for $\h$ too. This implies that the biorthogonal sequence to $(f_n)$ is complete in $\h$, which leads to a contradiction. 
\end{proof}

\begin{rem} 
	\label{exm_e1+en_bis}
	The sequence $(f_n)_{n=2}^\infty$ in Example \ref{exm_e1+en} is  minimal and complete in $\h$, and its biorthogonal sequence is $(e_n)_{n=2}^\infty$, which is not complete in $\h$. Hence,  Proposition \ref{weighted_biortho} gives another proof that $(f_n)_{n=2}^\infty$ is not a weighted frame for $\h$. 
\end{rem}

Further necessary conditions for the weighted frame property are given in the next proposition:

\begin{prop}
	\label{cond_weights}
	Let $(f_n)$ be a sequence in $\h$ and $(\omega_n)\subset (0,+\infty)$ be such that $(\omega_nf_n)$ is a Bessel sequence in $\h$. 
	Then 
	\begin{enumerate}
		\item $\sup_n(\omega_n \|f_n\|)<\infty$;
		\item if $\inf_n \|f_n\|>0$, then $\sup_n(\omega_n)<\infty$;  
		\item if $(f_n)$ is not a Bessel sequence, then $\inf_n(\omega_n)=0$. 
	\end{enumerate}
\end{prop}
\begin{proof}
	\begin{enumerate}
		\item[]
		Statement 1. comes from a well known result that the elements of a Bessel sequence are norm-bounded from above. Statement 2. follows trivially from 1. 
				
		\item[3.] Assume that $(f_n)$ is not a Bessel sequence and  $\inf_n(\omega_n)=c>0$. In this case there exists $f\in \h$ such that 
		$
		\sum_{n} |\ip{f}{f_n}|^2=+\infty,$
		which leads to a contradiction:
		\[
		\infty=c^2\sum_{n} |\ip{f}{f_n}|^2\leq \sum_{n} |\ip{f}{\omega_n f_n}|^2<\infty. \qedhere
		\]
	\end{enumerate}
	
\end{proof}

\begin{rem} 
	\label{exm_e1+en_ter}
	Let us come back again to the sequence $(f_n)_{n=2}^\infty$  from Example \ref{exm_e1+en}. As it was proved above in different ways,  $(f_n)_{n=2}^\infty$ is not a weighted frame for $\h$. 
	Here, we give another proof by applying 
	Proposition \ref{cond_weights}.  Assume that $(\omega_nf_n)_{n=2}^\infty$ is a frame for $\h$ for some $(\omega_n)\subset (0,+\infty)$. Since $(f_n)_{n=2}^\infty$ is not a Bessel sequence, $\inf_n (\omega_n)_{n=2}^\infty=0$ by Proposition \ref{cond_weights}, and thus there exists a subsequence  $(\omega_{n_k})$
	 such that $\displaystyle \lim_{k\to +\infty} \omega_{n_k}=0$. This leads to the following contradiction:   
	$$
	1=\|e_{n_k}\|^2\approx \sum_{n} \omega_{n}^2|\ip{e_{n_k}}{f_{n}}|^2=\omega_{n_k}^2 \to 0 \mbox{ as $k\to \infty$}. 
	$$  
\end{rem}

\subsection{Sufficient conditions}
\begin{prop}
	\label{sub_weight}
	If a sequence $(f_n)$ contains a proper subsequence which is a weighted frame for $\h$, then  $(f_n)$ is a weighted frame for $\h$.
\end{prop}
\begin{proof}
	Let us assume that $(f_n)=(f_{n_k})_{k\in I}\cup (g_m)_{m\in J}$ with some (not necessarily infinite) sequence $(g_m)_{m\in J}$ in $\h$ such that $(\omega_kf_{n_k})_{k\in I}$ is a frame for $\h$ with $(\omega_k)\subset (0,+\infty)$. Then, by Lemma \ref{lem_Bessel}, there exists a sequence $(\lambda_m)_{m\in J}\subset (0,+\infty)$ such that $(\lambda_mg_m)_{m\in J}$ is a Bessel sequence. Hence, $(\omega_kf_{n_k})_{k\in I}\cup(\lambda_mg_m)_{m\in J}$ is a frame for $\h$ and $(f_n)$ is a weighted frame for $\h$.	
\end{proof}

We recall that the excess of a sequence $(f_n)$  in $\h$ is defined as
$$e((f_n))=\sup\{|J|:J\subseteq I \text{ and } \overline{\text{span}}\{f_n\}_{n\in I\backslash J}=\overline{\text{span}}\{f_n\}_{n\in I}\},$$ 
and, moreover, a frame has null excess if and only if it is a Riesz basis (see, e.g., \cite[Section 8.7]{heil})\footnote{The excess of frames have been studied also in \cite{BCHL_excess,Beric,heil2,Holub}.}. The next statement gives a description of the class of weighted frames among all sequences with finite positive access.

\begin{prop}
	If a sequence in $\h$ has finite positive excess, then it is a weighted frame for $\h$ if and only if it contains a proper subsequence which is a weighted frame for $\h$ if and only if it contains a proper subsequence which is a weighted Riesz basis for $\h$.
\end{prop}
\begin{proof}
One of the directions follows from Proposition \ref{sub_weight}.  For the other direction, assume that $(f_n)$ is a weighted frame for $\h$ with finite positive excess. Then there exists $(\omega_n)\subset (0,+\infty)$ such that $(\omega_n f_n)$ is a frame for $\h$. Clearly, $(\omega_n f_n)$ has finite positive excess. Hence, by \cite[Lemma 8.42]{heil}, there exists a proper subsequence of $(\omega_n f_n)$ which is a Riesz basis for $\h$. Thus, $(f_n)$ contains a weighted Riesz basis.
\end{proof}

\subsection{Some characterizations of weighted frames}

Here, we consider more general operators than the analysis and synthesis ones. Given a sequence $\mathcal{F}=(f_n)$ in $\h$, we consider the operator $ \mathsf{C}_\mathcal{F}:\h \to \mathbb{C}^\mathbb{N}$ defined by $\mathsf{C}_\mathcal{F} f=(\ip{f}{f_n})$, and the operator $\mathsf{D}_\mathcal{F}:\dom(\mathsf{D}_\mathcal{F})\subseteq \mathbb{C}^\mathbb{N}\to \h$ determined by $\mathsf{D}_\mathcal{F}(c_n)=\sum_{n}c_n f_n$ for those complex sequences $(c_n)$ for which $\sum_{n}c_n f_n$ converges in $\h$.  Note that $\mathsf{C}_\mathcal{F}$ is a linear operator between vector spaces and it is equal to the analysis operator of  $(f_n)$ if and only if $(f_n)$ is a Bessel sequence. The following characterizations of the weighted frame property are inspired and based on known characterizations of the frame property \cite[Corollary 5.5.3 and Theorem 5.5.1]{ole1n}.

\begin{prop} Given a sequence $\mathcal{F}=(f_n)$ in $\h$, the following statements are equivalent.

1. $(f_n)$ 
 is a weighted frame for $\h$.

2.  There exists $(\omega_n)\subset (0,+\infty)$ such that $R(\mathsf{C}_\mathcal{F})$ is a Hilbert space with inner product $\ip{\cdot}{\cdot}_{\omega^2}$ defined by	
	$$
	\ip{(c_n)}{(d_n)}_{\omega^2} =\sum_{n}   \omega_n^2 c_n\overline{d_n},$$
and $\mathsf{C}_\mathcal{F}$ is a bounded and bijective operator between $\h$ and $R(\mathsf{C}_\mathcal{F})$.

3. There exists $(\omega_n)\subset (0,+\infty)$ such that the operator $\mathsf{D}_\mathcal{F}$ is well-defined on the Hilbert space $\ell^2_{\omega^{-2}}$ determined by	
	$$\ell^2_{\omega^{-2}}=\left \{(c_n)\in \mathbb{C}^\mathbb{N}: \sum_n \frac{|c_n|^2}{\omega_n^2}<\infty\right \}, \ \ \ 
\ip{(c_n)}{(d_n)}_{\omega^{-2}} =\sum_{n}   \frac{c_n\overline{d_n}}{\omega_n^2},
$$
	and  $R(\mathsf{D_\mathcal{F}})=\h$.
\end{prop}

\subsection{A conjecture about weighted frames and relation to unconditionally convergent multipliers}
\label{sec_multipliers}

Proposition \ref{duality_weighted} suggests the following conjecture:

\begin{conj}
	\label{conj_frames} A sequence  $(f_n)$ in $\h$ is a weighted frame for $\h$ if and only if there is a sequence $(g_n)$ so that for every $f\in\h$ one has 
	$f=\sum_n \<f,f_n\> g_n$ with unconditional convergence.
\end{conj}

First note that the representation $f=\sum_n \<f,f_n\> g_n$ in the above conjecture can equivalently be replaced by $f=\sum_n \<f,g_n\> f_n$, because  validity of $f=\sum_n \<f,f_n\> g_n$ with unconditional convergence for every $f\in\h$ is equivalent to validity of $f=\sum_n \<f,g_n\> f_n$ with unconditional convergence for every $f\in\h$ (see \cite[Lemmas 3.1 and 2.3]{uncconv2011}). 

As it was noticed in Prop. \ref{duality_weighted}, the ``only if'' direction of Conjecture \ref{conj_frames} holds true, so only the ``if'' direction is still open. 
The ``if'' direction can be related to conjectures and results on invertible unconditionally convergent multipliers in \cite{uncconv2011}. Let us first recall that a {\it multiplier} of two sequences $\Phi=(\varphi_n)$ and $\Psi=(\psi_n)$ in $\h$ and a  symbol (complex sequence) $m=(m_n)$ is the operator $M_{m,\Phi,\Psi}$ given by
\begin{equation}\label{multiplier}
	M_{m,\Phi,\Psi}f=\sum_{n} m_n\ip{f}{\psi_n} \varphi_n 
\end{equation}
for those $f$ from $\h$ for which the series in (\ref{multiplier}) converges. Multipliers were extensively studied in \cite{B2007,BS2015,Corso_mult1,Corso_mult2,Corso_mult3,SB2012,SB2013,uncconv2011,SB2014,SB2016,SB2017,SB2020} and generalized in different ways \cite{Mult_cont,Corso_distr_mult}. 
Note that a multiplier which is defined on the entire space $\h$ is automatically bounded on $\h$ {\cite{uncconv2011}}. 
A multiplier $M_{m,\Phi,\Psi}$ is said to be {\it unconditionally convergent} if the series in \eqref{multiplier} is unconditionally convergent for every $f\in \h$. 
The multiplier of two Bessel sequences and bounded symbol belongs to $\mathcal{B}(\h)$ and it is unconditionally convergent \cite{B2007}. In \cite{uncconv2011}, a main topic was a converse question, namely, the question whether an unconditionally convergent multiplier can be written as a multiplier of two Bessel sequences and symbol (1) just by shifting of weights. This question was affirmatively answered for some classes of multipliers \cite{uncconv2011, FGP} and the general case was posed as a conjecture:

\begin{conj1}[{\cite{uncconv2011}}]
	Let $M_{m,\Phi,\Psi}$ be an unconditionally convergent multiplier. Then there exist $(\alpha_n)\subset \mc$ and $(\beta_n)\subset \mc$  such that $(\alpha_{n}\varphi_n)$ and $(\beta_{n}\psi_n)$ are Bessel sequences and $\alpha_n\overline{\beta_n}=m_n$ for all $n$.
\end{conj1}

A question in a similar spirit, related to representation of an invertible multiplier as a multiplier of frames by shifting of weights, was also investigated in \cite{uncconv2011} and it can be stated as a conjecture too:

\begin{conj2}[{\cite{uncconv2011}}]
	Let $M_{m,\Phi,\Psi}$ be an unconditionally convergent multiplier which is invertible on $\h$. Then there exist $(\alpha_n)\subset \mc$ and $(\beta_n)\subset \mc$  such that $(\alpha_{n}\varphi_n)$ and $(\beta_{n}\psi_n)$ are frames for $\h$ and $\alpha_n\overline{\beta_n}=m_n$ for all $n$. 
\end{conj2}

The two conjectures (M1) and (M2) are actually related. 
That validity of (M1) implies validity of (M2) was observed already in \cite{uncconv2011}. 
Surprisingly, the converse implication turned out to also be true.  

\begin{thm}
	Conjectures (M1) and (M2) are equivalent. 
\end{thm}	
\begin{proof}	
	(M1) $\implies$ (M2) See  \cite[Section 4]{uncconv2011} for a proof. 
	
	(M2) $\implies$ (M1)  Let $M_{m,\Phi,\Psi}$ be an unconditionally convergent multiplier and thus a bounded operator on $\h$ as well. 
	Write $I-M_{m,\Phi,\Psi}=M_{(1),\mathcal{F}, \mathcal{G}}$ for some multiplier $M_{(1),\mathcal{F}, \mathcal{G}}$ with symbol $(1)$ and Bessel sequences $\mathcal{F}=(f_n)$ and $\mathcal{G}=(g_n)$ (for instance, take $\mathcal{G}$ to be an orthonormal basis for $\h$ and $f_n=(I-M_{m,\Phi,\Psi})g_n$ for every $n$). 
	Define $\Xi=(\xi_n)=(\varphi_1,f_1,\varphi_2,f_2,\dots)$, $\Theta=(\theta_n)=(\psi_1,g_1,\psi_2,g_2,\dots)$, and $m'=(m_n')=(m_1,1,m_2,1,\dots)$. 
	It is easy to see that $M_{m',\Xi,\Theta}$ is unconditionally convergent and equals to $I$. Then, if Conjecture (M2) is true, there exist $(\alpha_n)\subset \mc$ and $(\beta_n)\subset \mc$ such that  $(\alpha_n\xi_n)$ and $(\beta_n\theta_n)$ are frames for $\h$ and $\alpha_n\overline{\beta_n}=m_n'$ for every $n$, so $(\alpha_{2n-1}\varphi_n)$ and $(\beta_{2n-1}\psi_n)$ are Bessel sequences and $\alpha_{2n-1}\overline{\beta_{2n-1}}=m_{2n-1}'=m_n$ for all $n$.
\end{proof}

Let us now return to Conjecture 	\ref{conj_frames} and its relation to the above conjectures. Using the equivalence of definitions of weighted frames discussed at the beginning of Section \ref{sec4}, it is clear that validity of Conjecture (M2) implies validity of the ``if''-direction of Conjecture \ref{conj_frames} and thus 
one can state the following:

\begin{prop} \label{M2impliesNewconj}
	Conjecture (M2) implies Conjecture \ref{conj_frames}.
\end{prop}

It is still an open question whether the converse holds, i.e., whether validity of Conjecture \ref{conj_frames} implies validity of Conjecture (M2).

We end the section with some classes of sequences for which Conjecture  \ref{conj_frames} is true.

\begin{prop} Let $(f_n)$ and $(g_n)$ be sequences in $\h$ such that $\inf_{n} \|f_n\|\|g_n\|>0$ or $(g_n)$ is minimal or $(f_n)$ is minimal. The following statements are equivalent. 
	
	1. For every $f\in \h$, $
	f=\sum \limits_n \< f , f_n \>  g_n
	$
	with unconditional convergence. 
	
	2.  
	$(\|g_n\| f_n)$ and $(\|g_n\|^{-1} g_n)$ are dual frames for $\h$.
	
	3.  
	$(\|f_n\| g_n)$ and $(\|f_n\|^{-1} f_n)$ are dual frames for $\h$.
	
\end{prop}
\begin{proof}
	The implication $2. \Rightarrow 1.$ (resp. $3. \Rightarrow 1.$) follows from well known results in frame theory for any two sequences $(f_n)$ and $(g_n)$ with $g_n\neq 0$ for every $n$ (resp.  $(g_n)$ and $(f_n)$ with $f_n\neq 0$ for every $n$). 
	
	For the converse implications, assume that 1. holds.  
	If $\inf_{n} \|f_n\|\|g_n\|>0$, then  \cite[Corollary 3.6]{uncconv2011} implies that  	$(\|g_n\| f_n)$ and $(\|g_n\|^{-1} g_n)$  are Bessel sequences and then   \cite[Lemma 6.3.2]{ole1n} completes the proof that 2. holds. 
	If $(g_n)$ is minimal, then the proof of \cite[Proposition 4.8]{uncconv2011} implies that $\inf_{n} \|f_n\|\|g_n\|>0$  
and 
thus 2. holds by what is already proved. 
As mentioned at the beginning of Section \ref{sec_multipliers}, statement 1. is equivalent to validity of  
	$	f=\sum_n \< f , g_n \>  f_n	$
	with unconditional convergence for every $f\in\h$. Thus, 3. holds again from what is already proved. Finally, if $(f_n)$ is minimal, then 2. and 3. hold with similar arguments. 
\end{proof}

\section*{Acknowledgments} The first author thanks M. Speckbacher and D. Freeman for related discussions in the course of preparing ``Quantitative bounds for unconditional pairs of frames'', P. Balazs, D. Freeman, R. Popescu, M. Speckbacher, arXiv:2212.00947, and acknowledges the support by the project P 34624 ``Localized, Fusion and Tensors of Frames” (LoFT) of the Austrian Science Fund (FWF).  The second author thanks P. Balazs and D. Stoeva for the hospitality at the Acoustics Research Institute in 2019 and acknowledges partial supports by the Università degli Studi di Palermo (through FFR2023 ``Corso'') and by the ``Gruppo Nazionale per l'Analisi Matematica, la Probabilità e le loro Applicazioni'' (INdAM). The third author acknowledges support from the Austrian Science Fund (FWF) through Project P 35846-N ``Challenges in Frame Multiplier Theory''.


\end{document}